\documentclass[a4paper,12pt]{article} 

\usepackage[T1]{fontenc}
\usepackage{amssymb}

\usepackage[dvips]{graphicx}

\usepackage{amsfonts, amsmath, amsthm, amssymb}
\usepackage[ruled,vlined,linesnumbered]{algorithm2e}
\usepackage{epsfig}
\usepackage[T1]{fontenc}
\usepackage[cp1250]{inputenc}

\usepackage{tikz}
\usepackage{fancyhdr}
\makeindex

\allowdisplaybreaks[1]

\input{tcilatex}

\usepackage{graphicx}
\usepackage{amssymb}
\usepackage{amsmath}
\usepackage{amsthm}
\usepackage{tikz}
\usetikzlibrary{calc}

\theoremstyle{cuptheorem}
\newtheorem{prop}[theorem]{Proposition}
\newtheorem{cor}[theorem]{Corollary}
\theoremstyle{cupdefn}
\newtheorem{defn}[theorem]{Definition}
\theoremstyle{cuprem}
\newtheorem{rem}[theorem]{Remark}
\numberwithin{equation}{section}

\newcommand{\Lim}{\textup{Lim}}
\newcommand{\R}  {\mathop{\mathbb{R}}\nolimits}
\usepackage{diagrams}
\usepackage{tikz}

\theoremstyle{remark}

\numberwithin{equation}{section}

\begin{document}

\title{Inverse limits with countably Markov interval functions}
\author{Matev\v z \v Crepnjak, Tja\v sa Lunder}
\maketitle

\begin{abstract}
	We introduce countably Markov interval functions and show that two inverse limits with countably Markov interval bonding functions are homeomorphic 
	if the functions follow the same pattern. This result presents a generalization of well-known results of S.\ Holte, and I.\ Bani\v c and T.\ Lunder. 
\end{abstract}

\maketitle


%
%

\section{Introduction}

S.\ Holte proved under which conditions two inverse limits with Markov interval bonding functions are homeomorphic \cite{SHolte}. A generalization of Markov interval maps was introduced in \cite{BL}, where authors defined so-called generalized Markov interval functions. They are a non-trivial generalization of single-valued mappings from $I=[x,y]$ to $I$ to set-valued functions from $I$ to $2^I$. 

An upper semicontinuous function $f$ from $I = [x, y]$ to $2^I$ is a generalized Markov interval function with respect to $A=\{a_0, a_1, \ldots , a_m \}$, if 
\begin{itemize}
\item[-] $a_0=x < a_1 < \ldots < a_m=y$,
\item[-] the image of each point in every component of $I \setminus A$ is a singleton and for any two different points in the same component  of $I \setminus A$ holds that the corresponding images are different, 
\item[-] the image of every point in $A$ is a closed interval (possibly degenerate) with both endpoints in $A$, 
\item[-] for each $\displaystyle j=0,1, \ldots, m-1: \lim_{x \uparrow a_{j+1}} f(x), \lim_{x \downarrow a_{j}} f(x) \in A$.
\end{itemize}  
I.\ Bani\v c and T.\ Lunder also introduced the conditions when such functions follow the same pattern and, with this notion, proved that two generalized inverse limits with generalized Markov interval bonding functions are homeomorphic, if the bonding functions follow the same pattern. 
More precisely, let $ \{f_n\}_{n=1}^{\infty}$ be a sequence of upper semicontinuous functions from $I=\left[ a_0,a_m \right]$ to $2^{I}$ with surjective graphs, which are all generalized Markov interval functions with respect to $A=\{a_0 , a_1 , \ldots , a_m \}$ and let $ \{g_n\}_{n=1}^{\infty}$ be a sequence of upper semicontinuous functions from $J=\left[ b_0,b_m \right]$ to $2^{J}$ with surjective graphs, which are all generalized Markov interval functions with respect to $B=\{b_0 , b_1 , \ldots , b_m \}$. If for each $n$, $f_n$ and $g_n$ are generalized Markov interval functions with the same pattern, then $\varprojlim\{I,f_n\}_{n=1}^{\infty}$ is homeomorphic to $\varprojlim\{J,g_n\}_{n=1}^{\infty}$. 

We point out that Markov interval maps and generalized Markov interval functions are both defined with respect to a finite set $A$.
In this article we extend the notion of these functions in such a way that $A$ is an infinite countable set. We will call them countably Markov interval functions.
There are many examples of functions that have already been studied intensively in various areas for different reasons, and can now be interpreted as such countably Markov interval functions that are defined with respect to an infinite countable set. For example, (skew) tent functions can be interpreted as such countably Markov interval functions. Also, one can find such a function in \cite[p.\ 17]{ingram3}  
(a more detailed description of this function is given in Example \ref{BENNET}).

Our main result says that two inverse limits with countably Markov interval functions are homeomorphic, if the bonding functions follow the same pattern. 
 
%
%
\section{Definitions and notation}
	Our definitions and notation mostly follow \cite{BL,ingram,nadler}.

  A set is {\it countable\/} if it is finite or of the cardinality $\aleph_0$.
  
	A {\it map\/} is a continuous function. In the case where $f: \mathbb{R} \rightarrow \mathbb{R}$ is a map and $a \in \mathbb{R}$, we use $\displaystyle \lim_{x \downarrow a} f(x)$ to denote the 
	{\em right-hand limit\/} and $\displaystyle \lim_{x \uparrow a} f(x)$ to denote the {\em left-hand limit\/} 
	of the function $f$ at the point $a$ (for more details see \cite[p.\ 83--95]{RUDIN}). In Section 3 we define a generalization of this notion to limits of set-valued functions.
	
	For a metric space $(X, d)$, for $r>0$ and for $a \in X$, 
	$B_{r}(a)=\{x \in X\ |\ d(x,a)<r\}$ denotes an open ball in $X$.

	For a compact metric space $X$, we denote by $2^X$ the set of all nonempty closed subsets of $X$. 

	If $f:X\rightarrow 2^Y$ is a function, then the {\em graph\/} of $f$, $\Gamma(f)$, 
	is defined as $\Gamma(f)=\{ (x,y)\in X\times Y\ |\ y \in f(x)\}$.

	A function $f : X\rightarrow 2^Y$ {\em has a surjective  graph}, 
	if for each $y\in Y$ there is an $x\in X$, such that $y\in f(x)$. 

	Let $f : X\rightarrow 2^Y$ be a function. If for each open set $V\subseteq Y$, 
	the set $\{x\in X \ | \  f(x) \subseteq V\}$ is open in $X$, then $f$ is an {\em upper semicontinuous\/} function 
	(abbreviated u.s.c.) from $X$ to $2^{Y}$.

	The following theorem is a well-known characterization of u.s.c.\ functions between compact metrics spaces 
	(for example, see \cite[p.~120, Theorem 2.1]{ingram}).

	\begin{theorem}\label{grafi}
		Let $X$ and $Y$ be compact metric spaces and $f:X\rightarrow 2^Y$ a function. 
		Then $f$ is u.s.c.\ if and only if its graph $\Gamma (f)$ is closed in $X\times Y$.
	\end{theorem} 
	
	In this paper we use the standard projections, $\pi_1, \pi_2:X \times X \rightarrow X$, $\pi_1(x,y)=x$ and $\pi_2(x,y)=y$.
	
	If $F:X\rightarrow 2^Y$ is a u.s.c.\ function, where for each $x \in X$, the image $F(x)$ is a singleton in $Y$; then we can interpret it as a 
	single-valued continuous function. Obviously, for any continuous function $f:X\rightarrow Y$, the function $F:X\rightarrow 2^Y$, 
	defined by $F(x)=\{f(x)\}$, is an u.s.c.\ function.
	In this special case, we say that $F$ is injective if $f$ is injective.

	Let $A$ be a subset of $X$ and let $f:X \rightarrow 2^Y$ be a function.
	The {\em restriction \/} of $f$ on the set $A$, $f|_{A}$, is the function from $A$ to $2^Y$ 
	such that $f|_{A}(x)=f(x)$ for every $x \in A$.

	
	Let $f:X\rightarrow 2^{Y}$ be a function. Then we say that $f$ is {\it single-valued at some point\/} $x \in X$ 
	if $f(x)$ consists of a single point. We also say that $f$ is {\it single-valued on some subspace\/} $Z \subseteq X$ if the above holds for each $x \in Z$.  

	A sequence $\{X_k,f_k\}_{k=1}^{\infty}$ of compact metric spaces $X_k$ 
	and u.s.c.\ functions $f_k:X_{k+1}\rightarrow 2^{X_{k}}$, is an {\em inverse sequence with u.s.c.\ bonding functions\/}. 

	The {\em inverse limit\/} of an inverse sequence $\{X_k, f_k\}_{k=1} ^\infty$ with u.s.c.\ bonding functions is defined 
	as the subspace of $\prod_{k=1}^\infty X_k$ of all points $(x_1, x_2, \ldots )$, such that $x_k \in f_k(x_{k+1})$ for each $k$.
	The inverse limit of an inverse sequence $\{X_k, f_k\}_{k=1} ^\infty$ is denoted by $\varprojlim\{X_k,f_k\}_{k=1}^{\infty}$.

	In this paper we deal only with the case when for each $k$, $X_k$ is a closed interval $I=[x,y]$ 
	and $f_k:I\rightarrow  2^{I}$. So, we denote the inverse limit simply by $\varprojlim\{I,f_k\}_{k=1}^{\infty}$.

	The concept of inverse limits of inverse sequences with u.s.c.\ bonding functions (also known as generalized inverse limits) was
	introduced by Mahavier in \cite{mah} and later by Ingram and Mahavier in \cite{ingram}. 
	Since then, inverse limits have appeared in many papers (more references can be found in \cite{ingram4,ingram3}).

%
%
	
\section{Countably Markov interval functions}In this section we introduce countably Markov interval functions and show some of their properties.

	There are many results about limits of sequences of sets in metric spaces, for example see \cite[p.\ 56]{nadler} or \cite[p.\ 17]{macias}, 
	where other references can be found. 
	In the beginning of this section, we use a generalization of the above results about limits of sequences of sets, 
	by dealing with limits of set-valued functions $f$, i.e.\ the left-hand limit $\displaystyle \Lim_{t \uparrow a} f(t)$ and
	the right-hand limit $\displaystyle \Lim_{t \downarrow a} f(t)$.
	They are defined in such a way that in the cases where $f$ can be interpreted as a single-valued function, 
	the above limits behave as standard limits for single-valued functions (if they exist), as introduced in the previous section. 

\begin{defn}\label{limsup}
Let $f : I=[x, y] \to 2^{I}$ be a set-valued function. We define the left-hand limit and the right-hand limit of $f$ at a point $a \in I$ as follows: 
\begin{align*}
	\displaystyle \Lim_{t \uparrow a} f(t) = & \{t \in I \ | \ \textrm{ for each } \varepsilon > 0 \textrm{ there exists a } z \in (a - \varepsilon, a ) \\ & \textrm{ such that } (\{z\} \times f(z))\cap B_{\varepsilon}(a,t) \neq \emptyset\}.
\end{align*}
 \begin{align*}
\displaystyle \Lim_{t \downarrow a} f(t) = & \{ t \in  I\ | \ \textrm{for each } \varepsilon > 0 \textrm{ there exists a } z \in (a, a + \varepsilon ) \\ & \textrm{ such that } (\{z\} \times f(z))\cap B_{\varepsilon}(a,t) \neq \emptyset\}.
\end{align*}
\end{defn}

\begin{rem}
 If $f : I=[x, y] \to 2^{I}$ is a set-valued function, then the limits $\displaystyle \Lim_{t \uparrow x} f(t) $ and $\displaystyle \Lim_{t \downarrow y} f(t)$ are empty sets.
\end{rem}

\begin{example}\label{xxx}
Let the set-valued function $f: [0,1] \to 2^{[0,1]}$ be defined by its graph, which is the union of the following sets:
\begin{enumerate}
	\item the straight line segment with endpoints 
		$(\frac{1}{2}-\frac{1}{2n},\frac{1}{3})$ and $(\frac{1}{2}-\frac{1}{2n+1},\frac{2}{3})$ 
		for each positive integer $n$,	
	\item the straight line segment with endpoints $(\frac{1}{2}, 0)$ to $(\frac{1}{2},1)$,
	\item the straight line segment with endpoints $(\frac{1}{2}, \frac{4}{5})$ to $(1,1)$,
\end{enumerate}
see Figure \ref{Lim}.

The left-hand limit of $f$ at the point $\frac{1}{2}$ is $\Lim_{t \uparrow \frac{1}{2}} f(t) = \left[\frac{1}{3},\frac{2}{3} \right]$ and the right-hand limit is $\Lim_{t \downarrow \frac{1}{2}} f(t) = \{\frac{4}{5}\}$. 

\begin{figure}[h!]
	\centering
		\includegraphics[width=13.em]{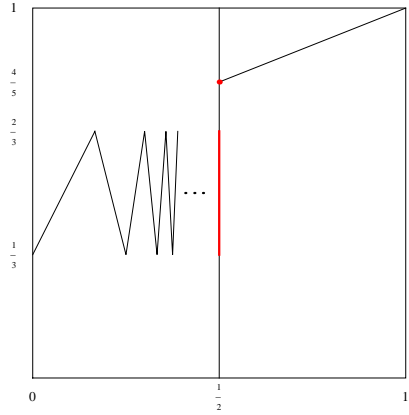}
		\caption{The graph of $f$ from Example \ref{xxx}.}
		\label{Lim}	
\end{figure}

\end{example}

\begin{example}\label{xxxx}
Let $g:I=[0,1] \to 2^{I}$ be defined by its graph,  $\Gamma(g)= ([0,\frac{1}{2})\times[0,1]) \cup ([\frac{1}{2},1]\times\{\frac{1}{2}\})$,
as seen in Figure \ref{Lim2}.
\begin{figure}[h!]
	\centering
		\includegraphics[width=13.em]{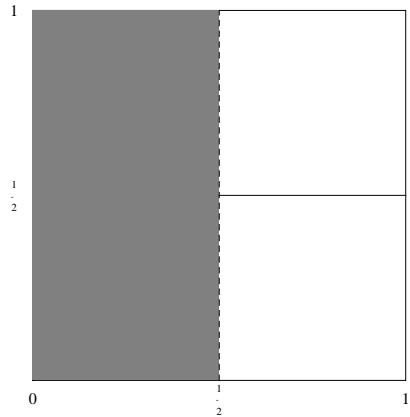}
		\caption{The graph of $g$ from Example \ref{xxxx}.}
		\label{Lim2}	
\end{figure}

The left-hand limit of $g$ at the point $\frac{1}{2}$ is  $\Lim_{t \uparrow \frac{1}{2}} g(t) = [0,1]$. 
The right-hand limit of $g$ at the point $\frac{1}{2}$ is $\Lim_{t \downarrow \frac{1}{2}} g(t) =  \{\frac{1}{2}\}$. 

\end{example}
The following auxiliary results mostly follow directly from Definition \ref{limsup} and are easy to proof. For the completeness of the paper, we give the proofs anyway.

\begin{lemma}\label{prvalema}
Let $x,y \in \R$, $x < y$. Let $f:I=[x,y] \rightarrow 2^{I}$ be a u.s.c.\ function and $a \in I$. The following two statements hold true:
\begin{enumerate}
	\item $\displaystyle \Lim_{t \uparrow a} f(t) \subseteq  f(a)$ and $\displaystyle \Lim_{t \uparrow a} f(t)$ is closed in $ f(a)$; if $a\neq x$ then $\displaystyle \Lim_{t \uparrow a} f(t) \neq \emptyset$.
	\item $\displaystyle \Lim_{t \downarrow a} f(t) \subseteq  f(a)$ and $\displaystyle \Lim_{t \downarrow a} f(t)$ is closed in $ f(a)$; if $a\neq y$ then $\displaystyle \Lim_{t \downarrow a} f(t) \neq \emptyset$.
\end{enumerate}
\end{lemma}

\begin{proof}
\begin{enumerate}
\item 
If $a = x$ then $\displaystyle \Lim_{t \uparrow a} f(t) = \emptyset$ and therefore the claim is true.

If $a \neq x$, then $\displaystyle \Lim_{t \uparrow a} f(t) \neq \emptyset$ -- 
this is easily seen since $[x,y]$ are metric compacta.
Now suppose that $\displaystyle \Lim_{t \uparrow a} f(t) \not \subseteq  f(a)$. 
This means that there exists a point $\displaystyle (a,s) \in \{a\} \times (\Lim_{t \uparrow a} f(t)  \setminus f(a))$. 
  From the definition of $\displaystyle \Lim_{t \uparrow a} f(t)$ it follows that there exists a convergent sequence $\{(x_{i},s_{i})\}_{i=1}^{\infty} \in \Gamma (f) \subseteq [x,y] \times [x,y]$ with the limit $(a,s)$. 
  Since $s_{i} \in f(x_{i})$ for each $i$ and $s \notin f(a)$, it follows that the graph of $f$ is not closed in $[x,y] \times [x,y]$ - a contradiction.
  
 Now we prove that $\displaystyle \Lim_{t \uparrow a} f(t)$ is closed in $ f(a)$. 
 
Let $\{(a,x_i)\}_{i=1}^{\infty}$ be a convergent sequence of elements in $\displaystyle \Lim_{t \uparrow a} f(t)$, with the limit $(a,x') \in \{a\} \times f(a)$.
We prove that $\displaystyle (a,x') \in \{a\} \times \Lim_{t \uparrow a} f(t) $.

Let $\varepsilon > 0$. 
Since $(a,x')$ is the limit of $\{(a,x_i)\}_{i=1}^{\infty}$, there exists an $i_0$ such that $(a,x_{i_0}) \in B_{\frac{\varepsilon}{2}}((a,x')) \subset [x,y] \times [x,y]$. Since $\displaystyle (a,x_{i_0}) \in \{a\} \times \Lim_{t \uparrow a} f(t)$, there exists a $z \in (a - {\frac{\varepsilon}{2}}, a )$ such that $ (\{z\} \times f(z))\cap B_{\frac{\varepsilon}{2}}((a,x_{i_0})) \neq \emptyset$. Therefore we can choose a point $(z,z')$ from $(\{z\} \times f(z))\cap B_{\frac{\varepsilon}{2}}((a,x_{i_0}))$.
 Then $$d((z,z'),(a,x')) \leq d((z,z'),(a,x_{i_0})) + d((a,x_{i_0}),(a,x')) \leq \frac{\varepsilon}{2} + \frac{\varepsilon}{2} = \varepsilon.$$ We have found a $z \in (a - \frac{\varepsilon}{2}, a ) \subseteq (a - \varepsilon, a )$, such that $(\{z\}\times f(z)) \cap B_{\varepsilon}((a,x')) \neq \emptyset$.

 \item With the same approach as in 1. we can prove that 
 	$\displaystyle \Lim_{t \downarrow a} f(t)$ is a closed subset of $f(a)$ and 
 	if $a\neq y$ then $\displaystyle \Lim_{t \downarrow a} f(t) \neq \emptyset$.
 \end{enumerate}
\end{proof}

\begin{lemma}\label{drugalema}
Let $x,y \in \R$, $x < y$. Let $f:I=[x,y] \rightarrow 2^{I}$ be a u.s.c.\ function. If $f(a)$ is connected for each $a \in I$, then the following statements hold true for each $a \in I$:
\begin{enumerate}
	\item If $a \neq x$, then $\displaystyle \Lim_{t \uparrow a} f(t) $ is nonempty and connected.
	\item If $a \neq y$, then $\displaystyle \Lim_{t \downarrow a} f(t) $ is nonempty and connected.
\end{enumerate}
\end{lemma}

\begin{proof}

	Note that $\Gamma(f)$ is a continuum by \cite[p.\ 17, Theorem 2.5.]{ingram4}.
	
\begin{enumerate}
	\item\label{prvaTocka} Choose any $a \in (x,y]$ and prove that $\displaystyle \Lim_{t \uparrow a} f(t)$ is connected
	(it is nonempty by Lemma \ref{prvalema}).
	  
	Assume that $\displaystyle \Lim_{t \uparrow a} f(t)$ is not connected. Then there exist nonempty open sets $U, V$ in $\displaystyle \Lim_{t \uparrow a} f(t)$ 
	such that $U \cap V \neq \emptyset$, $U \cup V = \displaystyle \Lim_{t \uparrow a} f(t)$. 
 
	Choose $v \in V$ and $u \in U$. Without loss of generality suppose that $u < v$.
	
	If $[u,v]\subseteq \displaystyle \Lim_{t \uparrow a} f(t)$, then $U \cap [u,v]$ and $V \cap [u,v]$ are nonempty, open in $[u,v]$ and 
	$U \cap V \cap [u,v] = \emptyset$ - this means that $[u,v]$ is not connected - a contradiction. Therefore there exist a $s \in (u,v) \backslash (U \cup V)$ 
	such that $(a,s) \in \{a\} \times f(a)$ (since $f(a)$ is connected by assumption).

	Note that it follows from Lemma \ref{prvalema} that $\displaystyle \Lim_{t \uparrow a} f(t)$ is closed in $ f(a)$. 
  Since $(a, u)$ and $(a, v)$ are points in $\displaystyle \Lim_{t \uparrow a} f(t)$, there exist sequences
  $$\{(x_{i}^1, y_i^1 )\}_{i=1}^{\infty} \in \Gamma (f)$$ with the limit $(a, u)$ and 
  $$\{(x_{i}^2, y_i^2 )\}_{i=1}^{\infty} \in \Gamma (f)$$ 
  with the limit $(a, v)$, and $x_i^1 < x_i^2 < x_{i+1} ^1$ for each positive integer $i$.
  Since $\displaystyle \lim_{i \to \infty} y_i ^1 = u$, $\displaystyle \lim_{i \to \infty} y_i ^2 = v$, and $u < s < v$, 
  there exist a positive integer $i_0$ such that for each $i \geq i_0$, $y_i^1 < s < y_i^2$.
  Recall that the graph of $f$ on each $[x_i ^1, x_i^2]$ is a continuum and therefore also connected by \cite[p.\ 17, Theorem 2.5.]{ingram4}. 
  Therefore we can choose a sequence $\{(x_i,s)\}_{i=i_0}^{\infty} \in \Gamma(f)$, where $x_i \in [x_i^1,x_i^2]$, such that $\{(x_i,s)\}_{i=1}^{\infty}$ is convergent with the limit $(a, s)$. This means that $(a, s) \in \displaystyle \{a\} \times \Lim_{t \uparrow a} f(t)$. Recall that $(a, s) \notin \{a\} \times (U \cup V)$. This means that $U \cup V \neq \displaystyle \Lim_{t \uparrow a} f(t)$.
  \item With the same approach as in \ref{prvaTocka}. we prove that $\displaystyle \Lim_{t \downarrow a} f(t) $ is nonempty and connected for each $a \in [x,y)$.
\end{enumerate}
	
\end{proof}

Finally, we introduce countably Markov interval functions.
Usually, a Markov interval function is defined with respect to a finite set $A$.
In the following definition we generalize the notion of Markov interval functions 
from the case where $A$ is finite to the case where $A$ is countable (including the case when $A$ is countably infinite). 

\begin{defn}\label{CMarkov}
Let $x,y\in \mathbb R$, $x<y$ and let $A$ be a countable subset of $I = \left[ x,y \right] $, containing the endpoints $x$ and $y$ and such that the derived set $A'$ of $A$ is a finite subset of $A$.
	We say that a u.s.c.\ function $f$ from $I$ to $2^I$ is a {\em countably Markov interval function with respect to $A$}, if
	\begin{enumerate}	 	
   \item\label{interval} for each $a \in A$, there exist $u,v \in A$ such that $u \leq v$ and $f(a)=[u,v]$ (degeneracy $u=v$ is possible),
	  \item\label{injektivna} the restriction of $f$ on every interval of $I\setminus A$ is an injective single-valued function,
	  \item for each $a \in A\setminus A'$, the limits $\displaystyle \Lim_{t \downarrow a} f(t), \Lim_{t \uparrow a} f(t)$ are subsets of $A$,
	  \item if $a \in A'$ and $a \neq x$, 
		then $\displaystyle \min(\Lim_{t \uparrow a} f(t)), \max(\Lim_{t \uparrow a} f(t)) \in A$;\\
			if $a \in A'$ and $ a \neq y$, then $\displaystyle \min(\Lim_{t \downarrow a} f(t)), \max(\Lim_{t \downarrow a} f(t)) \in A$.
	\end{enumerate}	
	The set $A$ is called the Markov partition for the function $f$.
	
	A function $f$ is a countably Markov interval function if there exists a set $A$, such that $f$ is countably Markov interval function with respect to $A$.
\end{defn}

	\begin{rem}
	 Since $A'$ is a subset of $A$, obviously $A$ is closed.
	 Minima and maxima that appear in Definition \ref{CMarkov} do exist since $\displaystyle \Lim_{t \uparrow a} f(t), \Lim_{t \downarrow a} f(t)$ 
	 are nonempty and closed by Lemma \ref{prvalema}.
	\end{rem}		

	\begin{rem}
		By Theorem \ref{grafi}, each countably Markov interval function has a closed graph, since it is a u.s.c.\ function.

		One can easily see that every (generalized) Markov interval function (as defined in \cite{SHolte} and \cite{BL}) 
		is also a countably Markov interval function. The set $A$ is finite, therefore also countable and the set $A'$ in this case is empty.
	\end{rem}
	
\begin{example}\label{BENNET}

	\begin{figure}[h!]
		\centering
			\includegraphics[width=13.em]{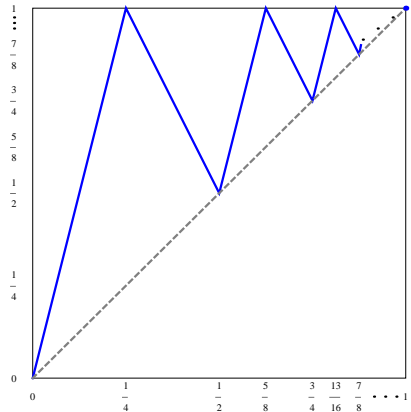}
			\caption{The graph of the function from Example \ref{BENNET}.}
			\label{bennet1}	
	\end{figure}
	
	Let $f: [0,1] \to 2^{[0,1]}$ be the function from \cite[p.\ 17]{ingram3}. 
	The graph $\Gamma(f)$ is the union of the infinite sequence of straight line
	segments as shown in Figure \ref{bennet1}. 
	Obviously, for $A=\{0,1\}\cup \{1 - \frac{3}{2^{i+1}}\ |\ i=1,2,3,\ldots\} \cup \{1-\frac{1}{2^i}\ |\ i=1,2,3,\ldots\}$, $f$ 
	is countably Markov interval function with respect to $A$. 
	The set $A'$ is the singleton $\{1\}$.
	
	This is an example of a countably Markov interval function whose inverse limit has already been studied;
	the example is taken from W.~T.~Ingram's book \cite[p.\ 17]{ingram3}, where it is attributed to R.~Bennet \cite{Bennet}.
	More details about its inverse limit can be found in the mentioned book.	
\end{example}	

In Lemma \ref{prvalema} and Lemma \ref{drugalema}, we have already shown some properties of $\displaystyle \Lim_{t \uparrow a} f(t)$ and $\displaystyle \Lim_{t \downarrow a} f(t)$ for any u.s.c.\ function $f$.
If $f$ is a countably Markov interval function with respect to $A$ and if $a \in A'$, we give a more precise description as follows.

\begin{prop}\label{limsup1}
	Let $f : I=[x, y] \to 2^{I}$ be a countably Markov interval function with respect to $A$. For each $a \in A'$ the following statement holds true.
	
	If $a \neq x$ then there exist $u_1, v_1 \in A$, $u_1 \leq v_1$, such that 
	$$\displaystyle \Lim_{t \uparrow a} f(t)  =  [u_1,v_1] \subseteq  f(a),$$
	and
	if $a \neq y$ then there exist $u_2, v_2 \in A$, $u_2 \leq v_2$, such that 
	$$\displaystyle \Lim_{t \downarrow a} f(t)  =  [u_2,v_2] \subseteq  f(a).$$
\end{prop}

\begin{proof}
	Let $a \in A'$, $a \neq x$.

	By Lemma \ref{prvalema}, $\displaystyle \Lim_{t \uparrow a} f(t) $ is a nonempty closed subset of $ f(a)$, 
	and by Lemma \ref{drugalema}, $\displaystyle \Lim_{t \uparrow a} f(t)$ is connected. 
	Therefore $\displaystyle \Lim_{t \uparrow a} f(t) =[u_1,v_1]$ for some $u_1,v_1 \in I$.
	By Definition \ref{CMarkov}, $u_1,v_1 \in A$, since $u_1 = \min(\Lim_{t \uparrow a} f(t))$ and $v_1 = \max(\Lim_{t \uparrow a} f(t))$.
	
	The claim about $\displaystyle \Lim_{t \downarrow a} f(t) $ can be proved analogously. 
\end{proof}


In the following definition we introduce when two countably Markov interval functions follow the same pattern.
This is done in such a way that for any two generalized Markov interval functions that follow the same pattern
(as introduced in \cite{BL}) also follow the same pattern (as introduced in Definition \ref{vzorec}) 
when being interpreted as countably Markov interval functions.

\begin{defn}\label{vzorec}
	Let $f: I=\left[ x,y \right] \rightarrow {2^I}$ be a countably Markov interval function with respect to $A$ and 
	let $g: J=\left[ x',y' \right] \rightarrow {2^J}$ be a countably Markov interval function with respect to $B$.
 We say that $f$ and $g$ {\em follow the same pattern with respect to $A$ and $B$} if there exists an increasing bijective function $\tau:A \rightarrow B$ such that
 for each $a \in A$ and for all $u,v\in A$,
	\begin{enumerate}
		\item $f(a)= [u,v] $ if and only if $g(\tau(a))= [\tau(u),\tau(v)] $,
		\item $\displaystyle \Lim_{t \uparrow a} f(t) = [u,v]$ if and only if $\displaystyle \Lim_{t \uparrow \tau(a)} g(t) = [\tau(u),\tau(v)]$, and
		\item $\displaystyle \Lim_{t \downarrow a} f(t) = [u,v]$ if and only if $\displaystyle \Lim_{t \downarrow \tau(a)} g(t) = [\tau(u),\tau(v)]$.	  
	\end{enumerate}
	
We say that $f$ and $g$ {\em follow the same pattern} if there exist Markov partitions $A$ and $B$, such that $f$ is countably Markov interval function with respect to $A$, $g$ is countably Markov interval function with respect to $B$, and 
$f$ and $g$ follow the same pattern with respect to $A$ and $B$.	
\end{defn}

\begin{rem}
	Note that in Definition \ref{vzorec} degeneration is possible, i.e.\ it may happen that $u=v$.
\end{rem}

In the following example we show that the function $\tau$ is not necessarily uniquely determined.

\begin{example}
	Let $f:[0,1] \to 2^{[0,1]}$ be a countably Markov interval function with respect to $A =  \{ \frac{1}{2n} \ | \ n = 1,2,3,\ldots  \} \cup \{1 -  \frac{1}{2(n+1)} \ | \ n = 1,2,3,\ldots \} \cup  \{0,1\}$,
	defined by its graph, which is the union of the following segments:	
	\begin{enumerate}
		\item the straight line segment with endpoints from $(a,0)$ to $(a,1)$ for each $a \in A$,
		\item the straight line segment with endpoints from $(a,0)$ to $(a',1)$ for all $a,a' \in A$, where $a<a'$ and $[a,a']\cap A = \{a,a'\}$.
	\end{enumerate}
	Obviously, $A' = \{0,1\}$.
	
	Let $g:[0,1] \to 2^{[0,1]}$, $\Gamma(g) = \Gamma(f)$ and let $B = A$.	
	
	Firstly, we define $\tau_1 : A \to B$ to be the identity function.  
	
	Secondly, we define $\tau_2 : A \to B$ by
	$$\tau_2(x) =
		\begin{cases}
			x \quad \text{;} & x = 0 \ \text{or} \ x = 1 \\
			\frac{3}{4} \quad \text{;} & x = \frac{1}{2} \\
			\frac{1}{2n}\quad \text{;} & x = \frac{1}{2(n+1)} \\
			1-\frac{1}{2(n+2)} \quad \text{;} & x = 1-\frac{1}{2(n+1)}, \\
		\end{cases}
	$$
	for each positive integer $n$.
	
	Obviously, $\tau_1$ and $\tau_2$ are both increasing bijective functions from $A$ to $B$ satisfying 1., 2.\ and 3.\ from Definition \ref{vzorec}.			
\end{example}


\section{Main result}

The following theorem is the main result of this paper.

\begin{theorem}\label{main}
	Let $ \{f_n\}_{n=1}^{\infty}$ be a sequence of functions from $I=\left[x,y\right]$ to $2^{I}$ 
	which are all countably Markov interval functions with respect to $A$ and let $ \{g_n\}_{n=1}^{\infty}$ 
	be a sequence of functions from $J=\left[x',y' \right]$ to ${2^J}$ which are all countably Markov interval functions 
	with respect to $B$. If $f_n$ and $g_n$ follow the same pattern with respect to $A $ and $B$ for each $n$, 
	then $\varprojlim\{I,f_n\}_{n=1}^{\infty}$ is homeomorphic to $\varprojlim\{J,g_n\}_{n=1}^{\infty}$.
\end{theorem}

\begin{proof}
	We fix an increasing bijective function $\tau:A \rightarrow B$, having the properties from Definition \ref{vzorec}.
	First we construct inductively homeomorphisms $h_i : I \to J$ such that $h_i(t) = \tau(t)$ for each $t\in A$ and $h_i \circ f_i = g_i \circ h_{i+1}$ 
	for each positive integer $i$,	see the diagram. 
	
	\begin{center}
		\begin{tikzpicture}[node distance=6em, auto]
			\node (X_1) {$I$};
			\node (X_2) [right of=X_1] {$I$};
			\node (X_3) [right of=X_2] {$I$};
			\node (X_4) [right of=X_3] {$I$};	
			\node (X_5) [right of=X_4] {$\ldots$};	
			\node (X_6) [right of=X_5] {$\varprojlim\{I,f_n\}_{n=1}^{\infty}$};	
			\draw[->] (X_2) to node {$f_1$} (X_1);
			\draw[->] (X_3) to node {$f_2$} (X_2);
			\draw[->] (X_4) to node {$f_3$} (X_3);
			\draw[->] (X_5) to node {$f_4$} (X_4);	
			\node (Y_1) [below of=X_1] {$J$};
			\node (Y_2) [right of=Y_1] {$J$};
			\node (Y_3) [right of=Y_2] {$J$};
			\node (Y_4) [right of=Y_3] {$I$};
			\node (Y_5) [right of=Y_4] {$\ldots$};
			\node (Y_6) [right of=Y_5] {$\varprojlim\{J,g_n\}_{n=1}^{\infty}$};	
			\draw[->] (Y_2) to node {$g_1$} (Y_1);
			\draw[->] (Y_3) to node {$g_2$} (Y_2);
			\draw[->] (Y_4) to node {$g_3$} (Y_3);
			\draw[->] (Y_5) to node {$g_4$} (Y_4);
			\draw[->] (X_1) to node {$h_1$} (Y_1);
			\draw[->] (X_2) to node {$h_2$} (Y_2);
			\draw[->] (X_3) to node {$h_3$} (Y_3);
			\draw[->] (X_4) to node {$h_4$} (Y_4);
			\draw[->] (X_6) to node {$H$} (Y_6);
		\end{tikzpicture}
	\end{center}	
	
	For $i=1$ we fix any homeomorphism $h_1:I \to J$ such that for any $t \in A$, $h_1(t) = \tau(t)$ 
	and for any interval $[a,a']\subseteq I$, such that $a<a'$ and $[a,a'] \cap A = \{a,a'\}$,
	$h_1|_{[a,a']}$ is a continuous strictly increasing bijective function from $[a,a']$ to $[\tau(a),\tau(a')]$.	Obviously, such a homeomorphism exists.

	Assume that we have already defined $h_1,h_2,\ldots,h_i$ such that $h_k(t) = \tau(t)$ for each $t\in A$ and $h_k \circ f_k = g_k \circ h_{k+1}$ for each $1\leq k \leq i-1$.
	Now we construct $h_{i+1}$.	
	\begin{enumerate}
		\item If $t \in A$, then $h_{i+1} (t) = \tau(t)$.
		\item If $t \notin A$, then there exist $a,a' \in A$, such that $[a,a'] \cap A = \{a,a'\}$ and $t \in (a,a')$.
			Since $f_i$ is a countably Markov interval function with respect to $A$,
			it follows that $f_i|_{(a,a')}$ is injective single-valued.
			Therefore $ \Lim_{s\downarrow a} f_i(s)$ and $ \Lim_{s\uparrow a'} f_i(s)$ are singletons in $A$, say, 
			$\{\alpha\} = \Lim_{s\downarrow a} f_i(s)$ and $\{\alpha'\} = \Lim_{s\uparrow a'} f_i(s)$.
			Therefore, if $f_i$ is increasing on $(a,a')$, then $f_i(t) \in (\alpha,\alpha')$, and if $f_i$ is decreasing on $(a,a')$, 
			then $f_i(t) \in (\alpha',\alpha)$.  Without loss of generality, we assume
			that $f_i$ is increasing on $(a,a')$. 			
			
			Obviously, $h_i(f_i(t)) \in (\tau(\alpha),\tau(\alpha'))$.
			Therefore 
			$$\emptyset \neq g_i^{-1}|_{(\tau(a),\tau(a'))} (h_i(f_i(t))) \subseteq (\tau(a),\tau(a')).$$ 
			Since $g_i|_{(\tau(a),\tau(a'))}:(\tau(a),\tau(a')) \to (\tau(\alpha),\tau(\alpha'))$ is bijective, 
			it follows that $|(g_i|_{(\tau(a),\tau(a'))})^{-1} (h_i(f_i(t)))| = 1$, say that $(g_i|_{(\tau(a),\tau(a'))}) ^{-1} (h_i(f_i(t))) = \{t'\}$.	So, we define 
			$$h_{i+1}(t) = t'.$$				
	\end{enumerate}

	It follows from the construction of $h_{i+1}$ that $h_{i+1}(t) = \tau(t)$ for each $t\in A$ and $h_i \circ f_i = g_i \circ h_{i+1}$.
	Now we prove that $h_{i+1}:I \to J$ is a homeomorphism.

	To prove that $h_{i+1}$ is bijective, it is enough to prove that $h_{i+1}|_{(a,a')}$ is bijective for each $a,a'\in A$, $[a,a'] \cap A = \{a,a'\}$. 
	The function $h_{i+1}|_{(a,a')}$ is defined as the composition of three bijective functions, namely, $f_{i}|_{(a,a')}$,
	$h_{i}|_{(\alpha,\alpha')}$ and $(g_{i}|_{(\tau(a),\tau(a'))})^{-1}$.
	Therefore $h_{i+1}$ is bijective.
	
	Since $h_{i+1}$ is strictly increasing and surjective, it is homeomorphism.	
	
	Now we define $H: \varprojlim\{I,f_n\}_{n=1}^{\infty} \to \varprojlim\{J,g_n\}_{n=1}^{\infty}$.
	
	Let ${\textbf{x}} = (x_1,x_2,x_3,\ldots) \in \varprojlim\{I,f_n\}_{n=1}^{\infty}$. Then
	$H({\textbf{x}})$ is defined by 
	$$H({\textbf{x}}) = (h_1(x_1),h_2(x_2),h_3(x_3),\ldots).$$
	By \cite[Theorem 4.5.]{{ingram4}}, $H$ is a homeomorphism.

\end{proof}

\begin{cor}
Let $f: I=\left[ x,y \right] \rightarrow 2^{I}$ be a countably Markov interval function with respect to $A$ and let $g: J=\left[ x',y' \right] \rightarrow 2^{J}$ be a countably Markov interval function with respect to $B$. If $f$ and $g$ are countably Markov interval functions with the same pattern then $\varprojlim\{I,f\}_{n=1}^{\infty}$ is homeomorphic to $\varprojlim\{J,g\}_{n=1}^{\infty}$.
\end{cor}

Matev\v z \v Crepnjak and Tja\v sa Lunder

(matevz.crepnjak@um.si and tjasa.lunder@gmail.com)

Faculty of Natural Sciences and Mathematics, 

University of Maribor, 

Koro\v{s}ka 160, SI-2000 Maribor, Slovenia

\end{document}